\documentclass{amsart}
\usepackage{amsmath,amssymb,bm}
\usepackage{color}
\usepackage{amsmath}
\usepackage{amssymb}
\usepackage{amscd}
\usepackage{amsthm}
\usepackage[dvips]{graphicx}

\theoremstyle{definition}
\newtheorem{defn}{Definition}[section]
\newtheorem{example}[defn]{Example}

\theoremstyle{plain}
\newtheorem{thm}[defn]{Theorem}
\newtheorem{prop}[defn]{Proposition}
\newtheorem{lem}[defn]{Lemma}

\newtheorem{question}[defn]{Question}

\newcommand{\Mdeg}{\operatorname{maxdeg}}
\newcommand{\rank}{\operatorname{rank}}

\newcommand{\KH}{\operatorname{KH}}

\numberwithin{equation}{section}

\title[]{The behavior of the maximal degree of the Khovanov homology under twisting}
\author{Keiji Tagami}
\date{\today}
\address{
Department of Mathematics,
Tokyo Institute of Technology,
Oh-okayama, Meguro, Tokyo 152-8551, Japan
}
\email{tagami.k.aa@m.titech.ac.jp}
\begin{document}
\maketitle
\begin{abstract}
In this paper, we study an asymptotic behavior of the maximal homological degree of the non-zero Khovanov homology groups under twisting.
\end{abstract}
\section{Introduction}\label{intro}
In \cite{khovanov1}, for each oriented link $L$, Khovanov defined a graded chain complex whose graded Euler characteristic is equal to the Jones polynomial of $L$. Its homology groups are link invariants called the Khovanov homology groups. 
Throughout this paper, we only consider the rational Khovanov homology. 
The Khovanov homology has two gradings, homological degree $i$ and $q$-grading $j$. 
By $\KH^{i}(L)$, we denote the homological degree $i$ term of the Khovanov homology groups of a link $L$ and by $\KH^{i,j}(L)$, we denote the homological degree $i$ and $q$-grading $j$ term of the Khovanov homology groups of $L$. 
\par
The maximal homological degree of the non-zero Khovanov homology groups of a link gives a lower bound of the minimal positive crossing number of the link (see Proposition~$\ref{i_max}$). The minimal positive crossing number of a link is the minimal number of the positive crossings of diagrams of the link. 
From this fact, it seems that the Khovanov homology estimates the positivity of links. 
\par
Sto{\v s}i{\'c} \cite[Theorem~$2$]{stosic2} showed that the maximal homological degree of the non-zero Khovanov homology groups of the $(2k, 2kn)$-torus link is $2k^{2}n$. 
By using the same method as Sto{\v s}i{\'c}'s, the author \cite[Corollary~$1.2$]{tagami1} proved that the maximal homological degree of the non-zero Khovanov homology groups of the $(2k+1, (2k+1)n)$-torus link is $2k(k+1)n$. 
These results intimate that the maximal degree of the non-zero Khovanov homology groups grows as the number of full-twists grows. 
\par
Let $L$ be an oriented link and $C_{p}$ be a disk which intersects with $L$ at $p$ points transversely with the same orientations as in Figure~$\ref{twist2}$. 
Then, for any positive integer $n$, we define a link $t_{n}(L;C_{p})$ as the link obtained from $L$ by adding $n$ full-twists at $C_{p}$. 

\begin{figure}[h!]
\begin{center}
\scalebox{0.4}{\includegraphics{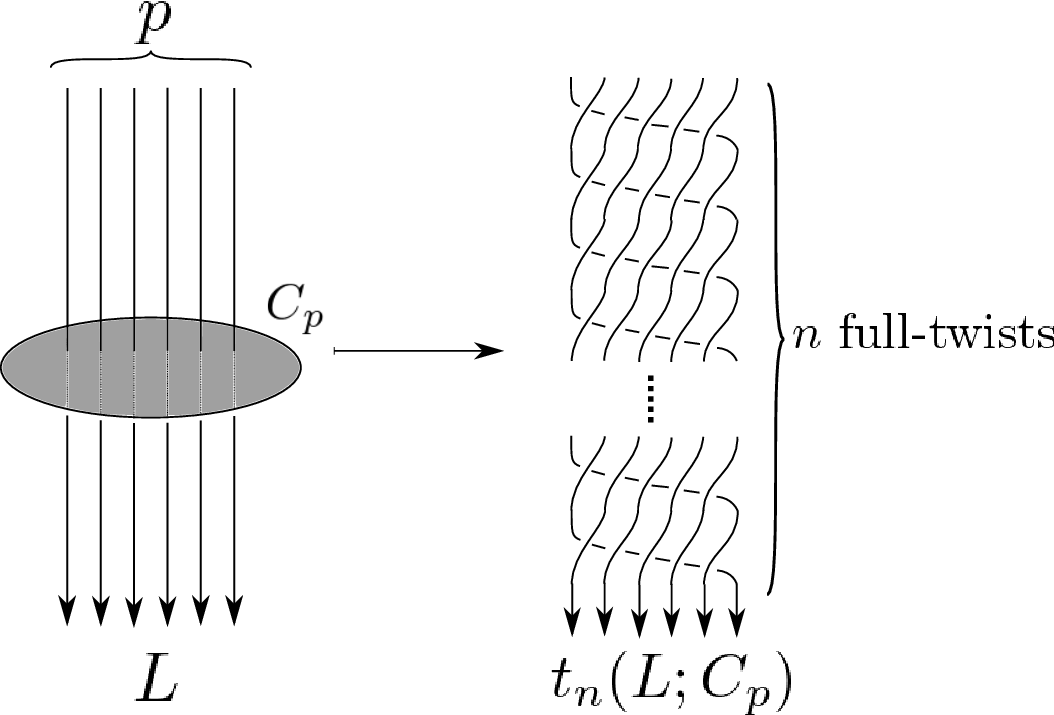}}
\end{center}
\caption{$t_{n}(L;C_{p})$. }
\label{twist2}
\end{figure}
\par
We consider the following question. 
\begin{question}\label{conjecture}
Let $L$ and $C_{p}$ be as above, and let $T_{p,pn}$ denote the positive $(p,pn)$-torus link. Does the following equality hold?
\begin{align*}
\lim_{n\rightarrow \infty}\frac{\max\{i\in\mathbf{Z}\mid\KH^{i}(t_{n}(L;C_{p}))\neq 0\}}{n}&=\lim_{n\rightarrow \infty}\frac{\max\{i\in\mathbf{Z}\mid\KH^{i}(T_{p, pn})\neq 0\}}{n}.
\end{align*}
\end{question}
Note that Sto{\v s}i{\'c} and we proved 
\begin{align}
\lim_{n\rightarrow \infty}\frac{\max\{i\in\mathbf{Z}\mid\KH^{i}(T_{p, pn}))\neq 0\}}{n}
=\begin{cases}
2k^{2}& \text{if\ \ }p=2k, \\
2k(k+1)&\text{if\ \ }p=2k+1. \label{torus}\\
\end{cases}
\end{align} 
\par
We proved the following, providing evidence towards an affirmative answer to Question~$\ref{conjecture}$. 
\begin{thm}[{\cite[Theorem~$1.3$ and Proposition~$1.4$]{tagami1}}]\label{tagami_cable}
Let $K$ be an oriented knot. 
Denote the $(p, pn)$-cabling of $K$ by $K(p,pn)$ for positive integers $p$ and $n$. 
Assume that each component of $K(p, pn)$ has an orientation induced by $K$, that is, each component of $K(p,pn)$ is homologous to $K$ in the tubular neighborhood of $K$. 
Put $c_{+}(K):=\min\{c_{+}(D)\mid D\text{\ is a diagram of \ }K\}$, where $c_{+}(D)$ is the number of the positive crossings of $D$. 
If $n\geq 2c_{+}(K)$, then we have the following for any positive integer $k$. 
\begin{align*} 
&\max\{i\in\mathbf{Z}\mid\KH^{i}(K(2k, 2kn))\neq 0\}=2k^{2}n, \\
2k(k+1)n\leq &\max\{i\in\mathbf{Z}\mid\KH^{i}(K(2k+1, (2k+1)n))\neq 0\}
\leq 2k(k+1)n+c_{+}(K). \\
\end{align*}
In particular, we have 
\begin{align*}
\lim_{n\rightarrow \infty}\frac{\max\{i\in\mathbf{Z}\mid\KH^{i}(K(p, pn))\neq 0\}}{n}=
\begin{cases}
2k^{2}& \text{if \ }p=2k, \\
2k(k+1)& \text{if \ }p=2k+1. 
\end{cases}
\end{align*}
\end{thm}
\par
In this paper, we consider the Question~$\ref{conjecture}$ for $p=2$. 
Precisely we prove the following. 
\begin{thm}[Main Theorem]\label{mainthm}
Let $L$ be an oriented link and $C$ be a disk which intersects $L$ at two points with the same orientations as in Figure~$\ref{twist}$. Then we have
\begin{center}
$\displaystyle{\lim_{n\rightarrow \infty}\frac{\max\{i\in\mathbf{Z}\mid \KH^{i}(t_{n/2}(L;C))\neq 0\}}{n}=1}$, 
\end{center}
where $t_{n/2}(L;C)$ is the link obtained from $L$ by adding $n$ half-twists at $C$. 
\end{thm}
\begin{figure}[h!]
\begin{center}
\scalebox{0.6}{\includegraphics{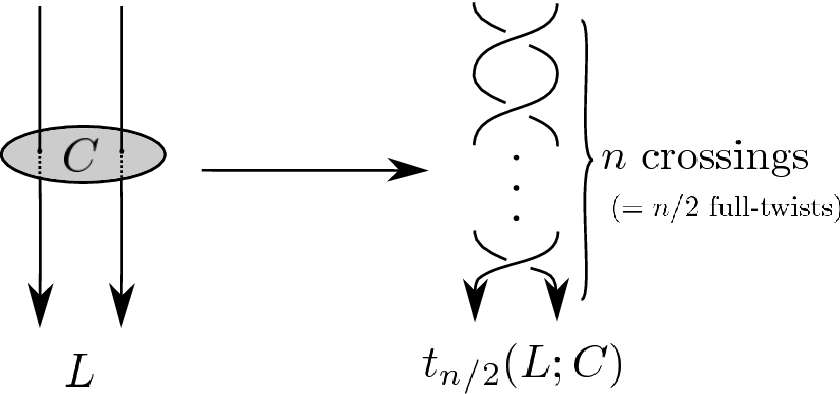}}
\end{center}
\caption{Definition of the link $t_{n/2}(L;C)$. The integer $n$ is the number of added half-twists. }
\label{twist}
\end{figure}
\par
There is Watson's work \cite{Watson1} on an asymptotic behavior of the reduced Khovanov homology under twisting. 
Ignoring grading, he gave a nice relation between the reduced Khovanov homologies of $t_{N/2}(L;C)$ and $t_{(N+1)/2}(L;C)$ for sufficiently large $N$ (\cite[Lemma $4.14$]{Watson1}). 
\par
This paper is organized as follows: 
In Section $\ref{def}$, we recall the definition of the Khovanov homology. 
In Section~$\ref{main}$, we prove our main theorem (Theorem~$\ref{mainthm}$). 
\section{Khovanov homology}\label{def}
\subsection{The definition of Khovanov homology}
In this subsection, we recall the definition of the (rational) Khovanov homology. 
Let $L$ be an oriented link. 
Take a diagram $D$ of $L$ and an ordering of the crossings of $D$. 
For each crossing of $D$, we define $0$-smoothing and $1$-smoothing as in Figure~$\ref{smoothing}$. 
A smoothing of $D$ is a diagram where each crossing of $D$ is changed by either $0$-smoothing or $1$-smoothing. 
\begin{figure}[!h]
\begin{center}
\scalebox{0.7}{\includegraphics{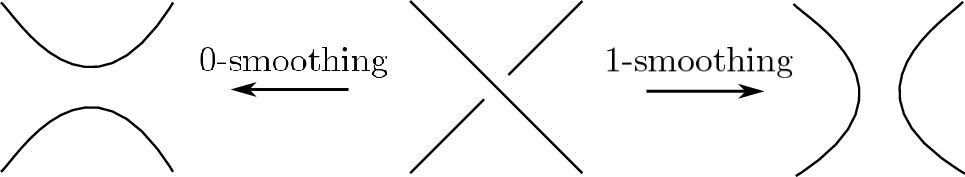}}
\end{center}
\caption{0-smoothing and 1-smoothing. }
\label{smoothing}
\end{figure}
Let $n$ be the number of the crossings of $D$. Then $D$ has $2^{n}$ smoothings. 
By using the given ordering of the crossings of $D$, we have a natural bijection between the set of smoothings of $D$ and the set $\{0, 1\}^{n}$, where, to any $\varepsilon =(\varepsilon _{1}, \dots, \varepsilon _{n})\in \{0, 1\}^{n}$, we associate the smoothing $D_{\varepsilon }$ where the $i$-th crossing of $D$ is $\varepsilon_{i}$-smoothed. 
Each smoothing $D_{\varepsilon }$ is a collection of disjoint circles. 
\par
Let $V$ be a graded free $\mathbf{Q}$-module generated by $1$ and $X$ with $\operatorname{deg}(1)=1$ and $\operatorname{deg}(X)=-1$. 
Let $k_{\varepsilon }$ be the number of the circles of the smoothing $D_{\varepsilon }$. 
Put $M_{\varepsilon }=V^{\otimes k_{\varepsilon }}$. 
The module $M_{\varepsilon }$ has a graded module structure, that is, for $v=v_{1}\otimes\cdots\otimes v_{k_{\varepsilon }}\in M_{\varepsilon }$, $\deg(v):=\deg(v_{1})+\cdots+\deg(v_{k_{\varepsilon }})$. 
Then define 
\begin{align*}
C^{i}(D)&:=\bigoplus_{|\varepsilon |=i }M_{\varepsilon }\{i\},  
\end{align*}
where $|\varepsilon |=\sum_{i=1}^{m}\varepsilon _{i}$. 
Here, $M_{\varepsilon}\{i\}$ denotes $M_{\varepsilon}$ with its gradings shift by $i$ (for a graded module $M=\bigoplus_{j\in\mathbf{Z}}M^{j}$ and an integer $i$, we define the graded module $M\{i\}=\bigoplus_{j\in\mathbf{Z}}M\{i\}^{j}$ by $M\{i\}^{j}=M^{j-i}$). 
\par
The differential map $d^{i}\colon C^{i}(D)\rightarrow C^{i+1}(D)$ is defined as follows. 
Fix an ordering of the circles for each smoothing $D_{\varepsilon }$ and associate the $i$-th tensor factor of $M_{\varepsilon }$ to the $i$-th circle of $D_{\varepsilon }$.  
Take elements $\varepsilon$ and $\varepsilon ' \in \{0, 1\}^{n}$ such that $\varepsilon _{j}=0$ and $\varepsilon' _{j}=1$ for some $j$ and that $\varepsilon _{i}=\varepsilon' _{i}$ for any $i\neq j$. 
For such a pair $(\varepsilon , \varepsilon ')$, we will define a map $d_{\varepsilon \rightarrow \varepsilon '}\colon M_{\varepsilon }\rightarrow M_{\varepsilon '}$ as follows. 
\par
In the case where two circles of $D_{\varepsilon }$ merge into one circle of $D_{\varepsilon' }$,  the map $d_{\varepsilon \rightarrow \varepsilon '}$ is the identity on all factors except the tensor factors corresponding to the merged circles where it is a multiplication map $m\colon V\otimes V\rightarrow V$ given by: 
\begin{center}
$m(1\otimes 1)=1$,\  $m(1\otimes X)=m(X\otimes 1)=X$,\  $m(X\otimes X)=0$. 
\end{center}
\par
In the case where one circle of $D_{\varepsilon }$ splits into two circles of $D_{\varepsilon' }$,  the map $d_{\varepsilon \rightarrow \varepsilon '}$ is the identity on all factors except the tensor factor corresponding to the split circle where it is a comultiplication map $\Delta \colon V\rightarrow V\otimes V$ given by:
\begin{center}
$\Delta (1)=1\otimes X+X\otimes 1$,\  $\Delta (X)=X\otimes X$. 
\end{center}
\par
If there exist distinct integers $i$ and $j$ such that $\varepsilon _{i}\neq\varepsilon '_{i}$ and that $\varepsilon _{j}\neq\varepsilon '_{j}$, then define $d_{\varepsilon \rightarrow \varepsilon '}=0$. 
\par
In this setting, we define a map $d^{i}\colon C^{i}(D)\rightarrow C^{i+1}(D)$ by $\sum_{|\varepsilon|=i}d_{\varepsilon}^{i}$, where $d_{\varepsilon}^{i}\colon M_{\varepsilon}\rightarrow C^{i+1}(D)$ is defined by 
\begin{align*}
d^{i}(v):=\sum_{|\varepsilon'|=i+1 }(-1)^{l(\varepsilon, \varepsilon' )}d_{\varepsilon \rightarrow \varepsilon '}(v).  
\end{align*}
Here $v\in M_{\varepsilon }\subset C^{i}(D)$ and $l(\varepsilon, \varepsilon')$ is the number of $1$'s in front of (in our order) 
the factor of $\varepsilon$ which is different from $\varepsilon'$. 
\par
We can check that ($C^{i}(D)$, $d^{i}$) is a cochain complex and we denote its $i$-th homology group by $H^{i}(D)$. 
We call these the unnormalized homology groups of $D$. 
Since the map $d^{i}$ preserves the grading of $C^{i}(D)$, the group $H^{i}(D)$ has a graded structure $H^{i}(D)=\bigoplus_{j\in\mathbf{Z}}H^{i,j}(D)$ induced by that of $C^{i}(D)$. 
For any link diagram $D$, we define its Khovanov homology $\KH^{i, j}(D)$ by 
\begin{center}
$\KH^{i, j}(D)=H^{i+n_{-}, j-n_{+}+2n_{-}}(D)$, 
\end{center}
where $n_{+}$ and $n_{-}$ are the number of the positive and negative crossings of $D$, respectively. 
The grading $i$ is called the homological degree and $j$ is called the $q$-grading. %
\begin{thm}[\cite{Bar-Natan-1}, \cite{khovanov1}]\label{khovanov}
For any oriented link $L$ and a diagram $D$ of $L$, 
the homology group $\KH(D)$ is preserved under the Reidemeister moves. 
In this sense, we can denote $\KH(L)=\KH(D)$. 
Moreover the graded Euler characteristic of the homology $\KH(L)$ equals the Jones polynomial of $L$, that is, 
\begin{align*}
V_{L}(t)=(q+q^{-1})^{-1}\sum_{i, j\in\mathbf{Z}}(-1)^{i}q^{j}\rank{\KH^{i, j}(L)}\Big|_{q=-t^{\frac{1}{2}}}, 
\end{align*}
where $V_{L}(t)$ is the Jones polynomial of $L$. 
\end{thm}
The following is an immediate consequence of the definition. 
\begin{prop}\label{i_max}
For any oriented link $L$, we have  
\begin{center}
$\max\{i\in\mathbf{Z}\mid\KH^{i} (L)\neq 0\}\leq c_{+}(L)$, 
\end{center}
where $c_{+}(L)$ is defined in Theorem~$\ref{tagami_cable}$. 
\end{prop}
\begin{proof}
For any diagram $D$ of $L$, by the definition, we have $H^{i}(D)=0$ if $i>c(D)$, where $c(D)$ is the number of the crossings of $D$. 
Hence $\KH^{i}(L)=\KH^{i}(D)=0$ for $i>c_{+}(D)$. 
Since $\KH^{i}(L)$ does not depend on the choice of $D$, we have $\KH^{i}(L)=0$ for $i>c_{+}(K)$. 
\end{proof}
%
\begin{example}
For example, the Khovanov homology of the left-handed trefoil knot $K$ (depicted in Figure~$\ref{example}$) is given as follows. 
\begin{align*}
\KH^{i, j}(K)=
\begin{cases}
\mathbf{Q}& \text{if\ }(i, j)=(0, -1), (0, -3), (-2, -5), (-3, -9), \\
0& \text{otherwise}. 
\end{cases}
\end{align*}

\begin{figure}[!h]
\begin{center}
\scalebox{0.3}{\includegraphics{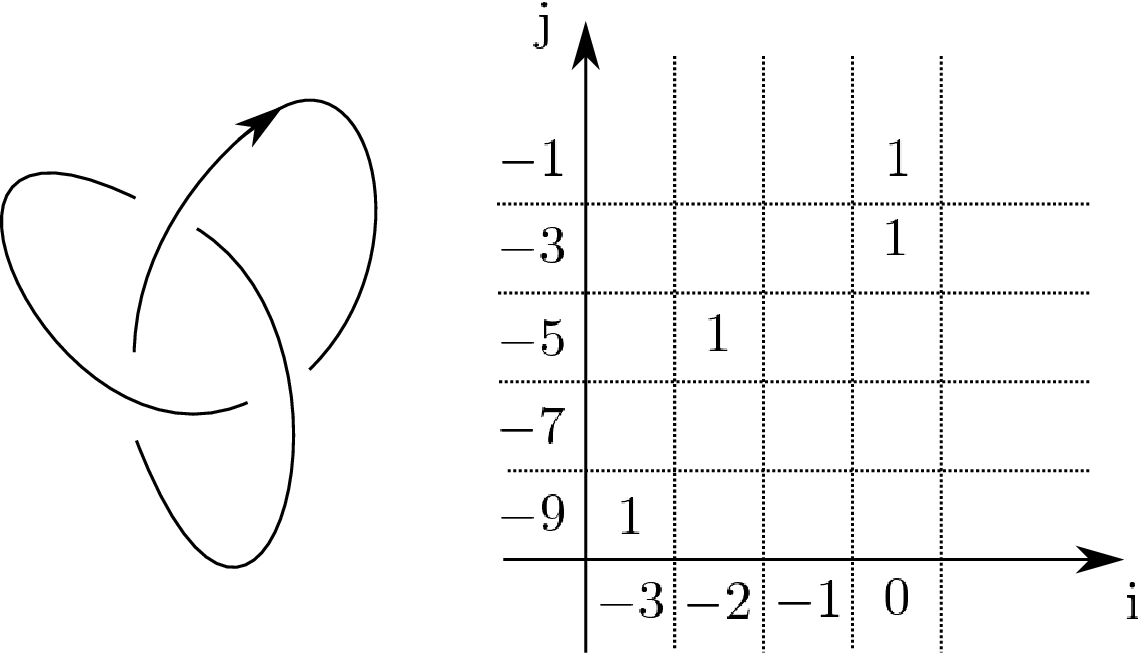}}
\end{center}
\caption{The left-handed trefoil and the table of the Khovanov homology of the knot whose $(i, j)$ element is $\dim_{\mathbf{Q}} \KH^{i, j}(K)$.  }
\label{example}
\end{figure}
\end{example}
%
%
\subsection{Main tool}
Our main tool is the following theorem proved by Wehrli \cite{spanning_kh} and Champanerkar and Kofman \cite{spanning2_kh}. 
The {\it {\bf 0}-smoothing} of a diagram $D$ is the disjoint circles obtained from $D$ by $0$-smoothing all crossings (see Figure~$\ref{smoothing}$). 
We define the {\it {\bf 1}-smoothing} of a diagram $D$ analogously. 
Then we have the following. 
\begin{thm}[\cite{spanning_kh}, \cite{spanning2_kh}]\label{spanning}
Let $D$ be a link diagram. If $H^{i, j}(D)\neq 0$, we have 
\begin{center}
$s_{1}(D)-2-c(D)\leq j-2i\leq 2-s_{0}(D)$, 
\end{center}
where $c(D)$ is the number of the crossings of $D$, and $s_{0}(D)$ and $s_{1}(D)$ are the numbers of the circles appearing in the {\bf 0}-smoothing and the {\bf 1}-smoothing of $D$, respectively. 
\end{thm}
\section{The main theorem and its proof}\label{main}
In this section, we prove our main theorem (Theorem~$\ref{mainthm}$). 
To prove this theorem, we first compute the maximal degree of the Jones polynomial and prove that it is proportional to the number $n$ of twists if $n$ is sufficiently large (Lemma~$\ref{lem1}$). 
Since the Jones polynomial is the graded Euler characteristic of the Khovanov homology, we obtain corresponding bounds on the gradings in which it is supported. 
\par
Let $L_{0}$ be an oriented link and $D_{0}$ be a diagram of $L_{0}$. 
Let $C$ be a disk as in Figure~$\ref{L_n}$ and $D_{n}$ be the diagram obtained from $D_{0}$ by adding $n$ half twists at $C$ as in Figure~$\ref{L_n}$. 
The diagram $D_{n}$ is a diagram of $t_{n/2}(L_{0};C)$. Put $L_{n}:=t_{n/2}(L_{0};C)$. 
\begin{figure}[!h]
\begin{center}
\scalebox{0.5}{\includegraphics{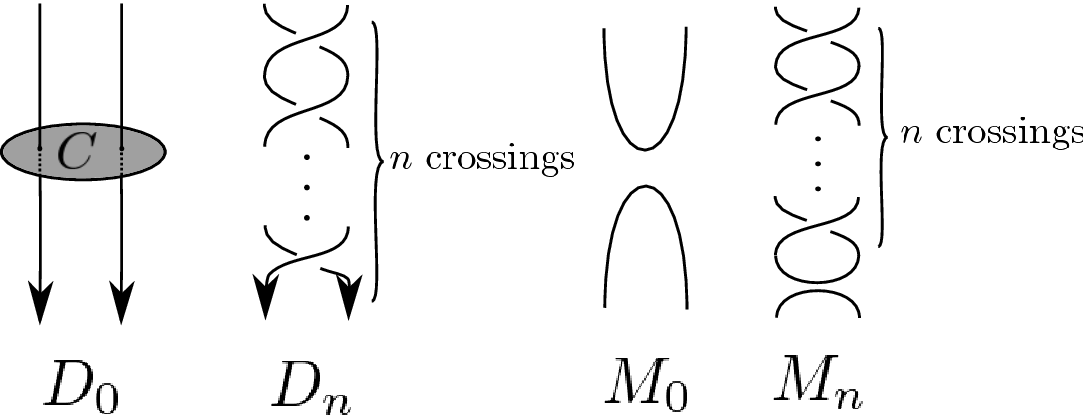}}
\end{center}
\caption{The diagrams $D_{0}$, $D_{n}$, $M_{0}$ and $M_{n}$. }
\label{L_n}
\end{figure}
\par
To prove Theorem~$\ref{mainthm}$, we use the following lemma. 
\begin{lem}\label{lem1}
There is a positive integer $N$ such that for any $n\geq N$
\begin{center}
$\Mdeg V_{L_{n+1}}(t)-\Mdeg V_{L_{n}}(t)= \frac{3}{2}$, 
\end{center}
where 
$\Mdeg V_{L}(t):=\max \{i\in\frac{1}{2}\mathbf{Z}\mid\text{the coefficient of\ }t^{i}\text{\ in\ }V_{L}(t) \text{\ is not zero}\}$. 
\end{lem}
\begin{proof}
The Kauffman bracket of an (unoriented) link diagram is given as follows:
\begin{itemize}
\item $\langle $ \raisebox{-0.8mm}{\scalebox{0.13}{\includegraphics{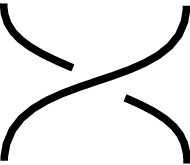}}} $\rangle=A\langle $ \raisebox{-1.2mm}{\scalebox{0.13}{\includegraphics{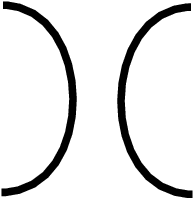}}} $ \rangle+A^{-1}\langle $
\raisebox{-1mm}{\scalebox{0.13}{\includegraphics{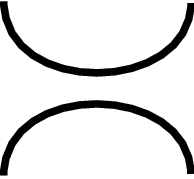}}} $ \rangle$, 
\item $\langle \bigcirc \rangle=1$, 
\item $\langle D \sqcup\bigcirc \rangle=(-A^{-2}-A^{2})\langle D \rangle$. 
\end{itemize}
For any link $L$, the Jones polynomial is given by 
\begin{center}
$V_{L}(t)=(-A)^{-3w(D)}\langle D \rangle \Big|_{A=t^{-\frac{1}{4}}}$, 
\end{center}
where $D$ is a diagram of $L$ and $w(D)$ is the writhe of $D$. 
Hence we have
\begin{align}
V_{L_{n}}(t)&=(-A)^{-3w(D_{n})}(A\langle D_{n-1} \rangle+A^{-1}\langle M_{n-1} \rangle)\Big|_{A=t^{-\frac{1}{4}}}\label{1}\\
&=(-A)^{-3w(D_{n-1})}(-A)^{-3}A\langle D_{n-1} \rangle\Big|_{A=t^{-\frac{1}{4}}}\nonumber\\
&\ \ \ +(-A)^{-3w(D_{n})}A^{-1}(-A)^{-3(n-1)}\langle M_{0} \rangle\Big|_{A=t^{-\frac{1}{4}}}\nonumber\\
&=-t^{\frac{1}{2}}V_{L_{n-1}}(t)+(-1)^{-3(w(D_{n})+n-1)}A^{-3(w(D_{0})+n)-3n+2}\langle M_{0} \rangle\Big|_{A=t^{-\frac{1}{4}}}\nonumber\\
&=-t^{\frac{1}{2}}V_{L_{n-1}}(t)+(-1)^{-3(w(D_{n})+n-1)}t^{\frac{3}{2}n}(A^{-3w(D_{0})+2}\langle M_{0} \rangle )\Big|_{A=t^{-\frac{1}{4}}}, \nonumber
\end{align}
where $M_{n}$ is the diagram depicted in Figure~$\ref{L_n}$. 
\par
By way of contradiction, assume that for any positive integer $n$, $\Mdeg V_{L_{n}}(t)-\Mdeg V_{L_{n-1}}(t)\leq  \frac{1}{2}$. Then we have $\Mdeg V_{L_{n}}(t)\leq  \frac{1}{2}n+\Mdeg V_{L_{0}}(t)$. 
Hence, from (\ref{1}), there is a positive integer $N'$ such that 
\begin{center}
$\Mdeg V_{L_{n}}(t)=\Mdeg (t^{\frac{3}{2}n}((A^{-3(w(D_{0})+2)}\langle M_{0} \rangle )\Big|_{A=t^{-\frac{1}{4}}})(t))$
\end{center}
for any $n\geq N'$. 
In particular, $\Mdeg V_{L_{N'+1}}(t)-\Mdeg V_{L_{N'}}(t)=\frac{3}{2}$. 
This is a contradiction. 
Hence, there is a positive integer $N$ such that 
\begin{align}
\Mdeg V_{L_{N}}(t)-\Mdeg V_{L_{N-1}}(t)> \frac{1}{2}. \label{aa} 
\end{align}
Then, we have 
\begin{align*}
\Mdeg V_{L_{N+1}}(t)-\Mdeg V_{L_{N}}(t)=\frac{3}{2}>\frac{1}{2}. 
\end{align*}
In fact, by the skein relation $t^{-1}V_{L_{+}}(t)-tV_{L_{-}}(t)=(t^{1/2}-t^{-1/2})V_{L_{0}}(t)$ for a usual skein triple $(L_{+}, L_{-}, L_{0})$, 
we obtain 
\begin{align*}
V_{L_{N+1}}(t)=t^{2}V_{L_{N-1}}(t)+(t^{\frac{3}{2}}-t^{\frac{1}{2}})V_{L_{N}}(t). 
\end{align*}
From $(\ref{aa})$, we have $\Mdeg V_{L_{N+1}}(t)=\Mdeg t^{\frac{3}{2}}V_{L_{N}}(t)=\frac{3}{2}+\Mdeg V_{L_{N}}(t). $
Inductively, we obtain $\Mdeg V_{L_{n+1}}(t)=\frac{3}{2}+\Mdeg V_{L_{n}}(t)$ for any $n\geq N$. 
\end{proof}
\begin{proof}[Proof of Theorem~$\ref{mainthm}$]
Let $L_{0}$, $L_{n}$, $D_{0}$ and $D_{n}$ be as above. 
Note that $s_{0}(D_{n})=s_{0}(D_{0})$, where $s_{0}(D_{0})$ and $s_{0}(D_{n})$ are introduced in Theorem~$\ref{spanning}$. 
By Theorem~$\ref{spanning}$, if $H^{i, j}(D_{n})\neq 0$, 
we have 
\begin{align}
j-2i\leq 2-s_{0}(D_{0}). \label{a}
\end{align}
\par
Since the graded Euler characteristic of the Khovanov homology is the Jones polynomial (Theorem~$\ref{khovanov}$), we obtain 
\begin{center}
$\max \{j\in\mathbf{Z}\mid\KH^{\ast, j}(D)\neq 0\}\geq 2\Mdeg V_{D}(t)+1$ 
\end{center}
for any link diagram $D$. Moreover 
\begin{align}
\max\{j\in\mathbf{Z}\mid H^{\ast, j}(D)\neq 0\}\geq 2\Mdeg V_{D}(t)+1-c_{+}(D)+2c_{-}(D), \label{b}
\end{align}
where $c_{+}(D)$ and $c_{-}(D)$ are the numbers of the positive and negative crossings of $D$, respectively. 
Put $f(D_{n}):=2\Mdeg V_{L_{n}}(t)+1-c_{+}(D_{n})+2c_{-}(D_{n})$ and $k(D_{n}):=\frac{1}{2}(f(D_{n})-2+s_{0}(D_{0}))$. 
From ($\ref{a}$), ($\ref{b}$), the definition of $H^{i}$ (the unnormalized Khovanov homology) and Figure~$\ref{supp}$, we have
\begin{align}
k(D_{n})\leq \max\{i\in\mathbf{Z}\mid H^{i}(D_{n})\neq 0\} \leq c(D_{n})=c(D_{0})+n. \label{c}
\end{align}
By Lemma~$\ref{lem1}$, there is a positive integer $N$ such that $k(D_{n+1})-k(D_{n})=1$ for any $n\geq N$. 
In particular $k(D_{n+1})=n-N+k(D_{N+1})$. 
From ($\ref{c}$), we have 
\begin{align*}
n-N+k(D_{N+1})\leq \max\{i\in\mathbf{Z}\mid \KH^{i}(L_{n})\neq 0\}+c_{-}(D_{0}) \leq c(D_{0})+n. 
\end{align*}
This implies Theorem~$\ref{mainthm}$. 
\end{proof}
\begin{figure}[!h]
\begin{center}
\scalebox{0.4}{\includegraphics{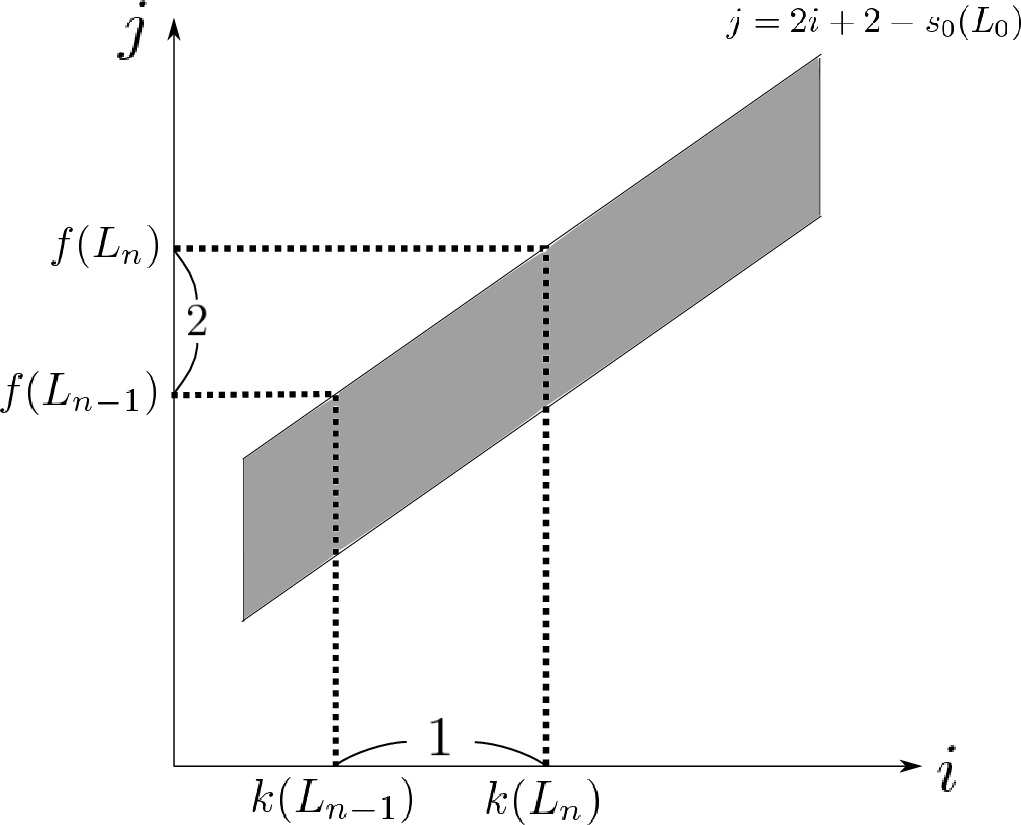}}
\end{center}
\caption{The gray parallelogram contains the support of $H^{i, j}$. }
\label{supp}
\end{figure}
%
%
\noindent{\bf Acknowledgements: }
The author would like to thank Hitoshi Murakami for his encouragements and helpful comments. 
He also would like to thank the referee. 
He was supported by JSPS KAKENHI Grant Number 25001362. 
%
%
%
\bibliographystyle{amsplain}
\bibliography{mrabbrev,tagami}
\end{document}